\documentclass[12pt]{article}
\usepackage{amssymb, amsmath,amsthm,graphicx,color,booktabs,hyperref,pdfsync}
\usepackage{graphicx}
\usepackage[lmargin = 3cm, rmargin = 3cm]{geometry}
\usepackage{pdfsync}
\usepackage{subfig}
\usepackage{multirow}
\usepackage{environ}
\usepackage{float}

\newcommand{\sump}{{{\sum}'}}
\newcommand{\abs}[1]{\left|#1\right|}
\newtheorem{theorem}{Theorem}[section]
\newtheorem{prop}{Proposition}
\newtheorem{corollary}{Corollary}[theorem]
\newtheorem{lemma}[theorem]{Lemma}
\newtheorem{remark}{Remark}
\newtheorem{definition}{Definition}[section]

\begin{document}
\title{The Fourier Cosine Method for Discrete Probability Distributions}

\author{
	Xiaoyu Shen\thanks{\href{https://fsquaredquant.nl}{FF Quant Advisory B.V.}, 3531 WR Utrecht, the Netherlands (\href{mailto:xiaoyu.shen@ffquant.nl}{xiaoyu.shen@ffquant.nl}).} 
	\and
	Fang Fang\thanks{Delft Institute of Applied Mathematics, Delft University of Technology, 2628 CD Delft, the Netherlands, and \href{https://fsquaredquant.nl}{FF Quant Advisory B.V.}, 3531 WR Utrecht, the Netherlands (\href{mailto:f.fang@tudelft.nl}{f.fang@tudelft.nl} and \href{mailto:fang.fang@ffquant.nl}{fang.fang@ffquant.nl}).}
	\and 
	Chengguang Liu\thanks{Delft Institute of Applied Mathematics, Delft University of Technology, 2628 CD Delft, the Netherlands, (\href{mailto:liucg92@gmail.com}{liucg92@gmail.com}).}
}

\maketitle

\abstract{We provide a rigorous convergence proof demonstrating that the well-known semi-analytical Fourier cosine (COS) formula for the inverse Fourier transform of continuous probability distributions can be extended to discrete probability distributions, with the help of spectral filters. We establish general convergence rates for these filters and further show that several classical spectral filters achieve convergence rates one order faster than previously recognized in the literature on the Gibbs phenomenon. Our numerical experiments corroborate the theoretical convergence results. Additionally, we illustrate the computational speed and accuracy of the discrete COS method with applications in computational statistics and quantitative finance. The theoretical and numerical results highlight the method's potential for solving problems involving discrete distributions, particularly when the characteristic function is known, allowing the discrete Fourier transform (DFT) to be bypassed.  }

\section{Introduction}
The Fourier cosine series (COS) method  (\cite{FangAndOosterleeCOS}) is an efficient semi-analytical solution to  the inverse Fourier transform problem. The key idea lies in the fact that the coefficients of a Fourier series expansion of a smooth probability density function can be analytically sampled from the characteristics function (ch.f.) up to an approximation error that can be made arbitrarily small. This allows for the avoidance of numerical methods such as the Fast Fourier Transform (FFT), which, while powerful, is not necessary when the Fourier transform is already available.

The COS method is known for its simplicity, speed, and accuracy. While the underlying principles and potential applications of COS extend beyond financial mathematics, its primary use till now has been in pricing financial derivatives and insurance products. To name only few among numerous applications, see \cite{FangAndOosterleeBarrier}, \cite{FangAndOosterleeHeston}, \cite{ZhangAndOosterleeAsian}, \cite{Jaber2018}, \cite{JaberAndDeCarvalho2024},  \cite{BondiPulidoScotti}, \cite{AndersenFusariTodorov}, \cite{Lietal}, and \cite{KANG202296}. Recently, COS has gained prominence as a computational tool for generating training data in deep learning studies in quantitative finance, owing to its exceptional speed and accuracy (see, e.g., \cite{GnoattoPicarelliReisinger} and \cite{risks7010016}). Beyond these, COS has also been utilized for solving backward stochastic differential equations \cite{COSBSDE}, coupled forward–backward stochastic differential equations \cite{HUIJSKENS2016593}, and stochastic control problems \cite{COSStochasticControl}. .

 To date, most studies on the COS method have assumed that the underlying probability distribution is continuous. The only attempt to extend the semi-analytical COS formula to discrete distributions is found in \cite{RuijterAndOosterlee2014}, though it offers only a heuristic argument. Another application of the COS method to discrete distributions is presented in \cite{SOUTOARIAS2025116177}, but it still relies on the Discrete Fourier Transform (DFT) for the inverse Fourier transform.  Our first contribution in this paper is to provide the first rigorous proof of convergence for the discrete cumulative distribution function (CDF) recovered by the same semi-analytical COS formula in \cite{RuijterAndOosterlee2014}, along with the corresponding error convergence rates. By establishing the theoretical convergence, we can now confidently bypass DFT and directly  invert the characteristics function (ch.f.) using the one-step semi-analytical COS formula to approximate discrete CDFs, in a similar way as we do for continuous distributions.

Additionally, we demonstrate how to recover the probability mass function (PMF) and compute moments using the semi-analytical COS formula, providing the associated error convergence rates. By doing so, we aim to extend the computational advantages of the COS method to a broader range of applications.
 
Although both \cite{RuijterAndOosterlee2014} and our work arrive at the same formula for the COS-recovered discrete CDF, the underlying approaches differ significantly. \cite{RuijterAndOosterlee2014} interprets this formula as a Riemann sum approximation of the L\'evy inversion formula  with a fixed integration step size. This interpretation would imply a first-order convergence due to the discretization error. In constrast, our derivation is  grounded in the convergence theory of spectral filters laid out in \cite{Vandeven1991}. We demonstrate that there is no numerical integration error in the COS-recovered CDF. Instead, the approximation error stems from the Gibbs phenomenon, with the convergence rate determined by the spectral filters employed to mitigate the Gibbs effect. The Gibbs phenomenon occurs because the sample space of a discrete distribution is inherently a discrete set.

Using Spectral filters is generally considered as a basic technique to reduce the impact of the Gibbs phenomenon. Nevertheless we choose spectral filters for three reasons. First, they offer the convenience as a mathematical tool for us to lay out the proof of the convergence that we need. Second, they preserve the simplicity and computational speed of the original COS method. Note that the use of spectral filters only requires modifying the COS coefficients analytically. It requires no extra computational cost, while other methods such as modifiers and adaptive filters (see e.g. \cite{GottbliebAndShu1991} and \cite{Tadmor2007}) are significantly more computationally expensive. Finally, the derived semi-analytical COS formula could serve as the first step for constructing the Gegenbauer polynomial approximation, which offers exponential convergence (see \cite{GottbliebAndShu1991} and \cite{Tadmor2007}) if higher accuracy is required at the expense of additional computational cost. Throughout this paper, we will stay with spectral filters, since our focus is on the extension of the original COS formula in \cite{FangAndOosterleeCOS} as a semi-analytical method.

In analyzing the convergence rate of the semi-analytical COS formula for discrete distributions, we uncovered new insights into the behavior of spectral filters. Specifically, we find that several classical spectral filters, including the Lanczos filter, raised cosine filter, sharpened raised cosine filter, second order exponential filters, exhibit a convergence rate that is one order higher than previously reported in \cite{Vandeven1991}. The convergence rate of these filters in \cite{Vandeven1991} is given as $O(K^{1-p})$, where $K$ refers to the number of Fourier terms and $p$ denotes the order of the filter. This result is also reflected in some highly review article and book, such as \cite{GottbliebAndShu1991} and \cite{Boyd}. However, in this paper, we demonstrate that the aforementioned spectral filters actually converge at a rate of $O(K^{-p})$.
 
 To showcase the potential applications, we provide two examples in computational statistics and quantitative finance. In a one-step calculation, we use COS to perform discrete Fourier inversion, and successfully recover the CDF of Poisson-binomial (PB) and generalized Poisson-binomial (GPB) distributions and a Hawkes process.
 
Throughout this paper, we assume that the discrete distribution has finite support. If the distribution's support is unbounded, a truncation range must be defined. A rule of thumb based on cumulants is proposed in \cite{FangAndOosterleeCOS}. More recently, a rigorous formula for the truncation range for continuous distributions has been introduced in \cite{JUNIKE2022126935} and \cite{Gero2024}. A technical formula for discrete distributions, analogous to the one in \cite{JUNIKE2022126935} and \cite{Gero2024}, can be derived in a similar manner. However, since the error convergence for the discrete distributions is dominated by the Gibbs impact, in this paper we omit this derivation to focus on our main objective.

The rest of this paper is structured as follows.  In Section \ref{sec:cos}, we present the convergence analysis of the semi-analytical COS formula for discrete distributions, extending it to bivariate discrete distributions as well. Section \ref{sec:conv_rate_special_filters} details the newly discovered convergence rates for the Lanczos filter, raised cosine filter, sharpened raised cosine filter, and second-order exponential filters. In Section \ref{sec:cos_moment}, we examine the convergence rate of the semi-analytical formula for moments obtained via the COS method. Section \ref{sec:numerical_example} provides numerical examples to validate the reported convergence rates and demonstrate applications. Finally, in Section \ref{sec:conclusion} we conclude.

\section{The COS method}\label{sec:cos}
For completeness, we begin by briefly reviewing the COS method for continuous distributions in Section \ref{sec:cos_continuous}. Next in Section \ref{sec:cos_discrete}, we discuss how to use the semi-analytical COS formula to approximate the CDF of univariate discrete distributions, providing proofs of convergence and convergence rates for the general case. Section \ref{sec:cos_bivariate} further extends the discrete COS method to bivariate discrete distributions, while in Section \ref{sec:cos_pmf}, we demonstrate how to recover the PMF using the discrete COS method.

\paragraph{Notation.}
We denote the CDF of a random variable $X$ as $F_X$. Depending on whether $X$ is a discrete or continuous random variable, $f_X$ represents either the PMF or PDF of $X$. The ch.f. of $X$ is denoted by $\varphi_X$. $i=\sqrt{-1}$ represents the imaginary unit, and the real part of a complex number is indicated by $\mathrm{Re}(\cdot)$.

\begin{definition}{} \label{def:spectral_filter}
Throughout this paper, we adhere to the definition of spectral filters provided in \cite{GottbliebAndShu1991}. A spectral filter of order $p$ is a real, even function $\sigma$ that satisfies the following properties:
	\begin{enumerate}
		\item $\sigma(0) = 1$,  $\sigma^l(0) = 0$ for $0 \leq l \leq p-1$.
		\item  $\sigma(\eta) = 0$ for $|\eta|> 1$.
		\item $\sigma(\eta) \in C^{p-1}$, $\eta \in (-\infty, +\infty)$.
	\end{enumerate}
\end{definition}

For example, the following spectral filters are widely recognized in the literature (see \cite{Hesthaven}, \cite{Gottlieb1992}). We will analyze their convergence behavior in detail later in Section \ref{sec:conv_rate_special_filters}.

\begin{itemize}
	\item Lanczos filter: $\sigma(\eta) = \sin(\pi \eta)/(\pi \eta)$.
	\item Raised cosine filter: $\sigma(\eta) = 1/2(1 + \cos(\pi \eta))$.
	\item Sharpened raised cosine filter:  $\sigma(\eta)= \sigma_r^4(\eta)(35 - 84 \sigma_2 (\eta) + 70  \sigma_r^2(\eta) - 20 \sigma_r^3(\eta)) $, where $\sigma_r$ refers to the raised cosine filter.
	\item Exponential filter: $\sigma(\eta)= e^{-\alpha\eta^p}$, where $p$ is an even integer and $\alpha=-\ln(\epsilon)$ with $\epsilon$ the machine epsilon.
\end{itemize}

\subsection{The COS method for continuous distributions}\label{sec:cos_continuous}
The COS method is an one-step, semi-analytical algorithm to solve the Fourier inversion problem. It approximates the PDF of a continuous random variable of $X$ on a truncated range $[a, b]$ as a Fourier cosine series expansion, i.e. for $a<x<b$,
\begin{equation}\label{eq:cos_pdf}
	f_X(x) \approx\sump_{k=0}^{K} A_k \cos\left( k\pi \frac{x- a}{b - a}\right),
\end{equation}
where $\sump$ indicates that the first term in the summation is weighted by one-half, and, the cosine coefficients $A_k$ can be directly sampled from the ch.f. of $X$ as follows:
\begin{equation}\label{eq:Ak}
	A_k =\frac{2}{b- a}\mathrm{Re}\left\{ \varphi_{X} \left(\frac{k\pi}{b - a}\right)\cdot \exp\left(-i\frac{k\pi a}{b - a}\right)\right\}.
\end{equation}
The truncation range $[a, b]$ is usually chosen carefully such that  $\mathrm{P}(X \notin [a, b])$ is sufficiently small. The COS method  can achieve exponential error convergence in the number of cosine terms for smooth enough probability densities such as Gaussian. 


\subsection{The COS method for univariate discrete distributions}\label{sec:cos_discrete}

%
%

Extending the COS method to discrete distributions requires addressing the challenge posed by the Gibbs phenomenon. In the literature (e.g., \cite{Vandeven1991} and \cite{GottbliebAndShu1991}), spectral filters are used to mitigate the effects of the Gibbs phenomenon by modifying the Fourier coefficients. This is achieved by multiplying the Fourier coefficients by the values of the sampled spectral filter. E.g., we could directly apply a filter to the original COS formula\ref{eq:cos_pdf}, which yields:

\begin{equation}\label{eq:filtered_cos_pmf}
	f_X^{\sigma}(x) = \sump_{k=0}^{K} A_k \sigma(k/K) \cos\left( k\pi \frac{x- a}{b - a}\right)  
\end{equation}

Note that the PMF $f_X$ is piecewise constant (equal to 0) function. Even though $f_X^{\sigma}(x)$ does converge to 0 wherever $f(x)=0$, it diverges at the discontinuities where $X$ takes value. Therefore, Equation \ref{eq:filtered_cos_pmf} is not very useful.

A solution is to integrate of the filter-adjusted COS PMF (Equation \ref{eq:filtered_cos_pmf}) to obtain the filter-adjusted COS CDF. The intuition is that the impact of discontinuities on the integration is sufficiently weak. However, it remains to be verified that the integral of the right hand side of Equation \ref{eq:filtered_cos_pmf} does preserve correctly the probability masses of $X$  at the discontinuities.To address this, we will next discuss the convergence. As a preliminary step, we first introduce a sequence of auxiliary functions $\mathcal{K}_l (x)$ which is borrowed from \cite{Vandeven1991}.

\begin{definition}{} \label{def:K_l}
	Following \cite{Vandeven1991}, we define a sequence of functions $\mathcal{K}_l(x)$ on $[0, 2\pi]$ as follows:
	\begin{align}
	\mathcal{K}_0 (x) = 1 + 2 \sum_{k = 1}^{K}\sigma(k/K)  \cos(kx), 	
	\end{align}
and for $l \geq 1$, \\
\begin{center}
\begin{math}
	\left\{
	\begin{array}{l}
		\mathcal{K}_l^{'} (x) = \mathcal{K}_{l-1}^{'} (x) \\
		\\
		\int_0^{2\pi}{K_l(x)dx} = 0,
	\end{array}
	\right.
\end{math}
\end{center}
with periodic extensions outside $[0, 2\pi]$.
\end{definition}

In particular, 
\begin{equation}\label{eq:K1}
\mathcal{K}_1(x) = x - \pi +\sum_{k = 1}^{K} \frac{2}{k} \sigma(k/K)  \sin(kx).
\end{equation}

To facilitate our analysis,  we list the following estimates of the upper bounds of  $K_l$  that are presented in \cite{Vandeven1991} and \cite{GottbliebAndShu1991} . 
Let $p$ be an integer equal to or greater than 2, the upper bounds in Inequality \ref{bound:K0}, \ref{bound:K_l_less_than_p}, and \ref{bound:K_p} hold. 
 For  $ x \in (0, 2\pi)$,  $K_0$ can be bounded as below:
\begin{equation}\label{bound:K0}
 \abs{K_0(x)} \leq C \lVert \sigma^{p} \rVert_{\mathrm{L}^2(0,1)} K^{1-p} (x^{-p} + (2\pi-x)^{-p}),
\end{equation}
for a positive constant $C$ independent of $K$, $p$, $\sigma$, and $x$.

For $1 \leq l \leq p$, $K_l$ admits similar bounds to $K_0$. Specifically, for $ x \in (0, 2\pi)$, 
\begin{align}\label{bound:K_l_less_than_p}
1 \leq l \leq p-1, \text{   } \abs{K_l(x)} \leq C_p  K^{1-p} (x^{-p} + (2\pi-x)^{-p}), 
\end{align}
and,
\begin{align}\label{bound:K_p}
l= p, \text{       } \abs{K_p(x)} \leq C_p  K^{1-p} (\abs{\ln(x) + \abs{\ln(2\pi -x)}}),
\end{align}
for a positive constant $C_p$ independent of $K$ and $x$.

\begin{theorem}\label{thm_cos}
Consider a discrete random variable $X$ with a finite number of possible values denoted by $\{\mathcal{X}_{1\leq m \leq M}\}$. Let $F_X$ represent the $\mathrm{CDF}$ of $X$. Without loss of generality, we assume $0 < \mathcal{X}_1 < \mathcal{X}_2 < \cdots < \mathcal{X}_M < \pi$.  For any $x \in (0,\pi)$ where $F_X$ is continuous, the filter-adjusted $\mathrm{COS}$ $\mathrm{CDF}$ $F_X^{\sigma}$  converges to  $F_X$ with a convergence rate of at least $O(K^{1-p})$, i.e.,
\begin{equation}\label{eq:weak_err_rate}
	|F_X^{\sigma}(x) - F_X(x)| \sim O(K^{1-p}),
\end{equation}
where 
\begin{equation}\label{eq:filtered_cos_cdf_pi}
	F_X^{\sigma}(x) = \frac{A_0}{2} x + \sum_{k=1}^{K}  \frac{ A_k\sigma(k/K) }{k} \sin( kx).
\end{equation}

Besides, the convergence error can be precisely expressed as:
\begin{align}\label{eq:error_filtered_cos_cdf}
	F_X^{\sigma}(x) -F_X(x)  \nonumber 
	=&     \sum_{m=1}^{M} \frac{ p_m}{2\pi} \left( \mathcal{K}_1(x+\mathcal{X}_m) +  \mathcal{K}_1(x- \mathcal{X}_m + 2\pi) \right) \mathbf{1}_{0 < x < \mathcal{X}_m} \nonumber \\
	&+ \sum_{m=1}^{M}  \frac{p_m }{2\pi} \left( \mathcal{K}_1(x-\mathcal{X}_m) +  \mathcal{K}_1(x + \mathcal{X}_m) \right) \mathbf{1}_{\mathcal{X}_m < x < \pi},
\end{align} 
where $p_m$ is the probability that $X = \mathcal{X}_m$, $\mathbf{1}_{\mathbf{A}}$ is the indicator function.

\end{theorem}

\begin{proof}
First, by the definition of $A_k$ in Equation \ref{eq:Ak}
\begin{equation}\label{eq:Ak2}
	A_k = \frac{2}{\pi} \mathrm{Re} \left\{ \varphi(k) \right\}
	= \frac{2}{\pi} \mathrm{Re} \left\{ \sum_{m=1}^{M} e^{i k \mathcal{X}_m} p_m \right\}
	= \frac{2}{\pi}  \sum_{m=1}^{M} \cos \left(k \mathcal{X}_m \right) p_m.
\end{equation}

Inserting Equation \ref{eq:Ak2} into Equation \ref{eq:filtered_cos_pmf} with $a=0$ and $b=\pi$, and changing the summation order yields

\[
	f_X^{\sigma}(x)  = \sum_{m=1}^{M} p_m  g_m^{K}(x),
\]
with
\[
	g_m^{K}(x) \equiv \frac{2}{\pi} \sump_{k = 0}^{K}  \sigma(k/K) \cos\left( k\mathcal{X}_m \right) \cos(kx).
\]
Integrating $f_X^{\sigma}$ from $0$ to $x$ gives
\[
F_X^{\sigma}(x) 
=  \sum_{m=1}^{M} p_m G_m^{K}(x),
\]
with
\[
	G^{K}_m(x) \equiv \frac{1}{\pi} x +   
	\sum_{k = 1}^{K} \frac{2}{k\pi} \sigma(k/K) \cos\left(k \mathcal{X}_m\right) \sin(kx).
\]

Since $F_X(x) = \sum_{m=1}^M p_m H_m(x)$, we can write
\begin{equation}\label{eq:error_decomp}
F_X^{\sigma}(x)- F_X(x) = \sum_{m=1}^M p_m (G^{K}_m(x) - H_m(x)). 
\end{equation}

Next, observe that on $[-\pi, \pi]$,
 \[ 
\sum_{k = 1}^{K} \frac{2}{k\pi}  \cos\left(k \mathcal{X}_m\right) \sin(kx)
\]  
 is the Fourier series expansion of $ -\frac{1}{\pi} x + H_m(x)$, where $H_m(x)$ is defined as
\begin{align*}
	H_m (x) =
	\begin{cases}
		1       & \quad \text{if } \mathcal{X}_m \leq x \leq \pi \\
		0       & \quad \text{if}  - \mathcal{X}_m < x < \mathcal{X}_m \\
		-1      & \quad \text{if } -\pi \leq x < -\mathcal{X}_m.
	\end{cases}
\end{align*}

%


Following similar derivations in the proof of the Proposition 3 of \cite{Vandeven1991}, it is straightforward to obtain
	\[
G^{K}_m(x) - H_m(x) = \begin{cases}
	\frac{1}{2\pi}  \left( \mathcal{K}_1(x+\mathcal{X}_m) +  \mathcal{K}_1(x- \mathcal{X}_m + 2\pi) \right) & \quad \text{if } 0 < x < \mathcal{X}_m \\
	\frac{1}{2\pi} \left( \mathcal{K}_1(x-\mathcal{X}_m) +  \mathcal{K}_1(x + \mathcal{X}_m) \right)  & \quad \text{if } \mathcal{X}_m < x < \pi.
\end{cases}
\]

Plugging the expressions above into Equation \ref{eq:error_decomp} yields Equation \ref{eq:error_filtered_cos_cdf}. Since $\mathcal{K}_1 \sim O(K^{1-p})$ for $x \in (0, 2\pi)$, it implies $ |F_X^{\sigma}(x) - F_X(x)| \sim O(K^{1-p})$.
\end{proof}

\begin{remark}\label{remark:extension_of_interval}
In Theorem \ref{thm_cos}, we assume that the $\mathrm{CDF}$ $F_X$ is right-continuous at $0$ and left-continuous at $\pi$. This assumption requires that the random variable $X$ does not take on values at the endpoints of the interval $[0, \pi]$. 

This assumption can always be satisfied. If there were any probability mass of $X$ at $0$ and/or $\pi$, we could extend the support of $X$ to $[-\Delta, \pi + \Delta]$ with $\Delta>0$ and then rescale $X$ to another variable $Y$ defined on $[0, \pi]$. Specifically, we can define $Y = \frac{\pi (X + \Delta)}{\pi + 2\Delta}.$
\end{remark}

The assumptions in Theorem \ref{thm_cos} can be adjusted to obtain the following two corollaries.

\begin{corollary}\label{thm_cos_finite_support}
	Consider a discrete random variable $X$ with a finite number of possible values $a < \mathcal{X}_1 < \mathcal{X}_2 < \cdots < \mathcal{X}_M < b$.  Let $F_X$ represent the $\mathrm{CDF}$ of $X$. For any $x \in (a, b)$ where $F_X$ is continuous, the filter-adjusted $\mathrm{COS}$ $\mathrm{CDF}$ $F_X^{\sigma}$  converges to  $F_X$ with a convergence rate of at least $O(K^{1-p})$, i.e.,
	\begin{equation*}
		|F_X^{\sigma}(x) - F_X(x)| \sim O(K^{1-p}),
	\end{equation*}
	where 
	\begin{equation}\label{eq:filtered_cos_cdf}
	F_{X}^{\sigma}(x)  = \frac{A_0}{2}( x - a ) + \sum_{k=1}^{K} A_k \sigma(k/K) \frac{ b - a }{k\pi} \sin\left( k\pi \frac{x- a}{ b - a}\right).
	\end{equation}
\end{corollary}

\begin{proof}
Proof of Corollary \ref{thm_cos_finite_support} is straightforward by considering a linear mapping from $X$ on $[a, b]$ to  $Y =\pi (X-a)/(b-a)$ on $[0, \pi]$.
\end{proof}

\begin{corollary}
	Consider a discrete random variable $X$ with countable possible values, i.e.,  $a < \mathcal{X}_{m} < b, m =1, 2, \cdots$.  Let $F_X$ represent the $\mathrm{CDF}$ of $X$. For any $x \in (a, b)$ where $F_X$ is continuous, the filter-adjusted $\mathrm{COS}$ $\mathrm{CDF}$ $F_X^{\sigma}$  converges to  $F_X$ with a convergence rate of at least $O(K^{1-p})$, i.e.,
	\begin{equation*}
		|F_X^{\sigma}(x) - F_X(x)| \sim O(K^{1-p}),
	\end{equation*}
	where 
	\begin{equation*}
		F_{X}^{\sigma}(x)  = \frac{A_0}{2}( x - a ) + \sum_{k=1}^{K} A_k \sigma(k/K) \frac{ b - a }{k\pi} \sin\left( k\pi \frac{x- a}{ b - a}\right).
	\end{equation*}
\end{corollary}

\begin{proof}
It suffices to prove the special case where $a=0$ and $b=\pi$. Furthermore, we can assume that $\limsup \mathcal{X}_{m} < b$ and $\liminf \mathcal{X}_{m} > a$. This assumption can always be met by extending and rescaling the support, as discussed in Remark \ref{remark:extension_of_interval}.
	
	In this case, the convergence error can be precisely expressed as:
	\begin{align}\label{eq:error_filtered_cos_cdf_countable_possibility}
		F_X^{\sigma}(x) -F_X(x)  \nonumber 
		=&     \sum_{m=1}^{+\infty} \frac{ p_m}{2\pi} \left( \mathcal{K}_1(x+\mathcal{X}_m) +  \mathcal{K}_1(x- \mathcal{X}_m + 2\pi) \right) \mathbf{1}_{0 < x < \mathcal{X}_m} \nonumber \\
		&+ \sum_{m=1}^{+\infty}  \frac{p_m }{2\pi} \left( \mathcal{K}_1(x-\mathcal{X}_m) +  \mathcal{K}_1(x + \mathcal{X}_m) \right) \mathbf{1}_{\mathcal{X}_m < x < \pi},
	\end{align}

Let $\limsup \mathcal{X}_{m}$ be denoted by $\bar{b}$ and $\liminf \mathcal{X}_{m}$ by $\bar{a}$. Thus
$\mathcal{K}_1(x+\mathcal{X}_m)$, $\mathcal{K}_1(x- \mathcal{X}_m + 2\pi)$, and $\mathcal{K}_1(x-\mathcal{X}_m)$ are bounded by $\max_{\bar{a} \leq s \leq \bar{b}} \abs{K_1(s)}$. Since $K_1 \sim O(K^{1-p})$ on $(0, 2\pi)$,  $|F_X^{\sigma}(x) - F_X(x)| \sim O(K^{1-p})$.
\end{proof}

\begin{remark}
	If the discrete random variable $X$ take values on an unbounded domain, e.g., $(-\infty, +\infty)$, we need truncate the domain to a large enough interval $[a, b]$.
	
	In order to recover the CDF from the ch.f., it is sufficient to choose the truncation range $[a, b]$ such as $\mathrm{P}(X \in [a, b])$ is sufficiently close to 1. E.g., we can numerically calculate the first two moments from ch.f. via the relation between the ch.f. and moments, and use the Chebyshev's inequality $\mathrm{P}(\abs{X - \mathrm{E}(X)} \geq c) \leq \mathrm{Var}(X)/c^2 $ to find a truncation range given the tolerance error.

\end{remark}

\subsection{Extension for bivariate discrete distributions}\label{sec:cos_bivariate}
The two-dimensional COS formula was initially introduced in \cite{RuijterAndOosterlee2012} for continuous random variables. In this section, we demonstrate that the results presented in Section \ref{sec:cos_discrete} can be extended for bivariate discrete distributions. The extension to higher dimensions follows a similar straightforward approach.

\begin{theorem}
	Consider a bivariate discrete random variable $X:=(X^1, X^2)$ with a finite number of possible values. We assume that each marginal $X^i$ is constrained to  $[0, \pi]$. That is, $X^i \in  \{ \mathcal{X}_1^i, \mathcal{X}_2^i , \cdots, \mathcal{X}_M^i\} $  and $\min_{1\leq m \leq M}{\mathcal{X}_m^i} > 0, \max_{1\leq m \leq M}{\mathcal{X}_m^i} <\pi$ for $i =1, 2$. Let $F_X$ represent the true joint $\mathrm{CDF}$ of $(X^1, X^2)$. On the domain where $F_X$ is continuous, the filter-adjusted $\mathrm{COS}$ $\mathrm{CDF}$ $F_X^{\sigma}$ converges to $F_X$ with a convergence rate of at least $ O(K_1^{1-p} K_2^{1-p})$, i.e.,
	\begin{equation}\label{eq:two_dim_err_rate}
		|F_X^{\sigma}(x_1, x_2 ) - F_X(x_1, x_2)| \sim O(K_1^{1-p} K_2^{1-p}),
	\end{equation}
	where 
	\begin{align}
		F_X^{\sigma}(x) & = \frac{x_1 x_2}{\pi^2} 
		+ \frac{1}{2} \sum_{k_1=1}^{K_1} \frac{A_{k_1, 0} \sigma(k_1/K_1) }{k_1}x_2\sin(k_1 x_1) 
		+ \frac{1}{2} \sum_{k_2=1}^{K_2} \frac{A_{0, k_2} \sigma(k_2/K_2) }{k_2}x_1\sin(k_2 x_2) \nonumber \\
		&+  \sum_{k_1=1}^{K_1} \sum_{k_2=1}^{K_2}  \frac{ A_{k_1,k_2}\sigma(k_1/K_1) \sigma(k_2/K_2) }{k_1 k_2} \sin( k_1x_1) \sin(k_2x_2). \label{eq:filtered_2d_cos_cdf_pi}
	\end{align}
$A_{k_1,k_2}$ are the two-dimensional COS coefficient that can be directly sampled from the ch.f. of $X$ as
\begin{equation}\label{eq:Ak_2D}
	A_{k_1, k_2} = \frac{1}{2}(A^{+}_{k_1, k_2} + A^{-}_{k_1, k_2}),
\end{equation}
with
\[
A^{\pm}_{k_1, k_2} = \frac{4}{\pi^2} \mathrm{Re}\left\{ \varphi_X(k_1, k_2) \cdot \exp\left(-ik_1 \mp ik_2\right) \right\}.
\]
\end{theorem}

\begin{proof}
The COS coefficients $A_{k_1, k_2}$ given in Equation \ref{eq:Ak_2D} can be rewritten as
\begin{equation}\label{eq:Ak2_2D}
	A_{k_1, k_2} = \frac{4}{\pi^2} \sum_{m_1=1}^M\sum_{m_2=1}^M p_{m_1, m_2} \cos(k_1 \mathcal{X}_{m_1}^1 ) \cos(k_2 \mathcal{X}_{m_2}^2 ),
\end{equation}

with $p_{m_1, m_2}$ is the probability that $X^1 = \mathcal{X}_{m_1}$ and $X^2 = \mathcal{X}_{m_2}$.\\

Plugging \ref{eq:Ak2_2D} into Equation \ref{eq:filtered_2d_cos_cdf_pi} and changing the summation order yields

\begin{align}
	F_X^{\sigma}(x_1, x_2) = \sum_{m_1=1}^{M} \sum_{m_2=1}^{M} p_{m_1, m_2} G_{m_1}^{K_1} (x_1) G_{m_2}^{K_2} (x_2),
\end{align}
with
\[
G^{K_i}_{m_i}(x) \equiv \frac{1}{\pi} x_i +   
\sum_{k = 1}^{K_i} \frac{2}{k\pi} \sigma(k/K_i) \cos\left(k \mathcal{X}^i_{m_i}\right) \sin(kx_i).
\]
Note that $F_X(x_1, x_2) = \sum_{m_1=1} \sum_{m_2=1} p_{m_1, m_2} H_{m_1}(x_1) H_ {m_2}(x_2)$. From the proof of Theorem \ref{thm_cos} we know $G^{K_i}_{m_i}(x) - H_{m_i}(x) \sim O(K_i^{1-p}))$. Thus we obtain the convergence rate in \ref{eq:two_dim_err_rate}. 
\end{proof}

\subsection{Recovery of Probability Mass Function}\label{sec:cos_pmf}
In certain applications, we possess information regarding the specific locations where the discrete random variable assumes values. This knowledge enables us to approximate the probability mass function (PMF) using the filter adjusted Fourier series.
E.g., under the assumptions made for Theorem \ref{thm_cos}, we can approximate $\mathrm{P}(X= \mathcal{X}_i)$ by $F_X^{\sigma}( \mathcal{X}_i + dx) - F_X^{\sigma}( \mathcal{X}_i - dx) $, where $dx$ is a small positive number.

\begin{theorem}\label{thm_cos}
Consider a discrete random variable $X$ with a finite number of possible values denoted by $\{\mathcal{X}_{0\leq m \leq M}\}$. Let $F_X$ represent the $\mathrm{CDF}$ of $X$. Without loss of generality, we assume $0 < \mathcal{X}_1 < \mathcal{X}_2 < \cdots < \mathcal{X}_M < \pi$. We can recover the $\mathrm{PMF}$ $f_X(x)$ with the filter adjusted Fourier series. That is, for $1 \leq i \leq M-1$,  
	\begin{equation}\label{eq:weak_err_rate}
		|F_X^{\sigma}(\mathcal{X}_i + dx ) - F_X^{\sigma}(\mathcal{X}_i - dx ) - f_X(x)| \sim O(dx) O(K^{1-p}),
	\end{equation}
	where $F_X^{\sigma}$ is given in Equation \ref{eq:filtered_cos_cdf_pi}.
\end{theorem}

\begin{proof}
	For sufficiently small $dx$, we have $\mathcal{X}_{i-1} < \mathcal{X}_i - dx < \mathcal{X}_i < \mathcal{X}_i + dx < \mathcal{X}_{i+1}$. Therefore
	\begin{align}
	&|F_X^{\sigma}(\mathcal{X}_i + dx ) - F_X^{\sigma}(\mathcal{X}_i - dx ) - f_X(x)|  \nonumber \\ 
	&= |( F_X^{\sigma}(\mathcal{X}_i + dx ) - F_X^{\sigma}(\mathcal{X}_i - dx ) ) - (F_X(\mathcal{X}_i + dx )  - F_X(\mathcal{X}_i - dx )|  \nonumber \\
	&= | (F_X^{\sigma}(\mathcal{X}_i + dx ) - F_X(\mathcal{X}_i + dx ) )  -  (F_X^{\sigma}(\mathcal{X}_i - dx ) - F_X(\mathcal{X}_i - dx )) | \nonumber  \\
	& \leq  \sum_{m=1}^{M} \frac{ p_m}{2\pi} | \mathcal{K}_1( \mathcal{X}_i +\mathcal{X}_m + dx) - \mathcal{K}_1( \mathcal{X}_i +\mathcal{X}_m - dx) | \nonumber \\
	&+ \sum_{m=1}^{M} \frac{ p_m}{2\pi} | \mathcal{K}_1( \mathcal{X}_i - \mathcal{X}_m + 2\pi + dx) - \mathcal{K}_1(\mathcal{X}_i - \mathcal{X}_m + 2\pi - dx) | \nonumber \\
	\end{align}
	where in the last step we make use of Equation \ref{eq:error_filtered_cos_cdf} and that $\mathcal{K}_1$ is $2\pi$ periodic.
	
	For  $x \in (0, 2\pi)$,  $\mathcal{K}_1(x)$  can be written as 
	\begin{equation}\label{eq:K02K1}
		\mathcal{K}_1(x) = \mathcal{K}_1(\pi) + \int_{\pi}^x \mathcal{K}_0 (t)dt = \int_{\pi}^x \mathcal{K}_0 (t)dt.
	\end{equation} 
Since $\mathcal{K}_0 \sim O(K^{1-p})$ for $x \in (0, 2\pi)$,  we obtain Equation \ref{eq:weak_err_rate}.
	
\end{proof}

\section{Superior convergence rates for some spectral filters}\label{sec:conv_rate_special_filters}
In this section, we present new results on the convergence rates for several popular spectral filters. A critical condition for obtaining the convergence rate of the filtered Fourier series in \cite{Vandeven1991} and \cite{GottbliebAndShu1991} is that $\mathcal{K}_0 \sim O(K^{1-p})$ for $x \in (0, 2\pi)$. This condition was proved in \cite{Vandeven1991} and \cite{GottbliebAndShu1991} using the Poisson summation formula and integration by parts. 


The Poisson summation formula states that 

\begin{equation}\label{eq:poisson_k0}
	\mathcal{K}_0(x) =  K \sum_{k=-\infty}^{\infty} \hat{s} (Kx + 2kK\pi),
\end{equation}
with
\begin{equation}\label{eq:s_hat}
 \hat{s}(y) =  \int_{-1}^{1} \sigma(x) e^{iyx}dx.
\end{equation}

Using the fact that $\mathcal{K}_1$ is the integral of $\mathcal{K}_0$,  \cite{Vandeven1991} and \cite{GottbliebAndShu1991} obtain $\mathcal{K}_1 \sim O(K^{1-p})$ for $x \in (0, 2\pi)$.


Rather than applying integration by parts, here we directly analyze the magnitude of $\hat{s}(Kx + 2kK\pi)$ for large values of $K$, demonstrating the superior convergence rate of $\mathcal{K}_0$ for various filters, including the Lanczos filter, raised cosine filter, sharpened raised cosine filter, and second-order exponential filter. This approach yields several new results:
\begin{itemize}
\item First, we establish that the Lanczos filter is a valid filter with a convergence rate of $\mathcal{K}_1 \sim O(K^{-1})$, even though it does not satisfy the definition of a spectral filter given in  \ref{def:spectral_filter}.
\item Second, we show that the convergence rate for the raised cosine filter, sharpened raised cosine filter, and second-order exponential filter is one order faster than previously reported in \cite{Vandeven1991} and \cite{GottbliebAndShu1991}, i.e., $\mathcal{K}_1 \sim O(K^{-p})$.
\end{itemize}

\subsection{The Lanczos filter}\label{sec:lanczos}
 Lanczos filter is considered in \cite{GottbliebAndShu1991} as a first-order filter only in a formal sense. Obliviously, the convergence formula  $\mathcal{K}_1 \sim O(K^{1-p})$ obtained in \cite{Vandeven1991} and \cite{GottbliebAndShu1991} indicates divergence of $\mathcal{K}_1$ for $p=1$. Thus, strictly speaking, the Lanczos filter does not belong to the class of spectral filters by the definition \ref{def:spectral_filter}. The proofs in  \cite{Vandeven1991} and \cite{GottbliebAndShu1991}, require $\sigma(\cdot)$ and its derivative vanishing at $\pm 1$. Therefore, these proofs are not suitable for the Lanczos filter, as the derivative is not 0 at $\pm 1$.

Below we prove that  $\mathcal{K}_0$ and  $\mathcal{K}_1$ converge as $O(K^{-1})$ for the Lanczos filter, thereby justifying that Lanczos is effectively a first-order filter. To show that $\mathcal{K}_1 \sim O(K^{-1})$ for Lanczos, we first establish the following lemma.

\begin{lemma}\label{theorem:conv_rate_transform_of_lanczos}
	Let $\sigma(x)= \frac{\sin(\pi x)}{\pi x}$.  We have the following inequality for $|x|\ge 2\pi $, i.e.,
	\begin{align*}
		| \int_{-1}^1e^{ix z}\sigma(z)dz | \le \frac{38\pi}{3 x^2}. 
	\end{align*} 
\end{lemma}

\begin{proof}
  Set $\varphi(y)=\int_{-1}^1e^{iy z}\sigma(z)dz.$ Since $e^{iy z}\sigma(z)$ is differentiable and integrable, we have for $\abs{y}\ge 2\pi, $
\begin{align*}
\frac{d\varphi}{dy}= \int_{-1}^1e^{iy z}iz\sigma(z)dz=\frac{-2\sin(y)}{(\pi+y)(\pi-y)}.
\end{align*}
Together with that $\lim_{\abs{x}\to \infty}\varphi(x)=0$, we obtain that for $x\ge 2\pi$
\begin{align*}
	\varphi(x) &= -\int_x^{+\infty} \frac{d\varphi}{dy}dy= \int_x^{+\infty}  \frac{2\sin(y)}{(\pi+y)(\pi-y)}dy\\
	&=  \int_x^{2n_x\pi}  \frac{2\sin(y)}{(\pi+y)(\pi-y)}dy+\sum^\infty_{i=0}  \int^{2(n_x+i+1)\pi}_{2(n_x+i)\pi}  \frac{2\sin(y)}{(\pi+y)(\pi-y)}dy\\
\end{align*}
with $n_x=[x/(2\pi)]+1$. \\
Because of $2\pi \leq 2(n_x-1)\pi\le x\le 2n_x\pi$, we can estimate a bound of the first term as follows:
\begin{align*}
	\abs{ \int_x^{2n_x\pi}  \frac{2\sin(y)}{(\pi+y)(\pi-y)}dy}\leq \frac{2\pi}{-\pi^2+x^2}\leq \frac{8\pi}{3x^2}.
\end{align*}
To bound the second term, note that
\begin{align*}
	\abs{\int^{2(n_x+i+1)\pi}_{2(n_x+i)\pi}  \frac{2\sin(y)}{(\pi+y)(\pi-y)}dy} &\leq \abs{\int^{2(n_x+i+1)\pi}_{2(n_x+i)\pi}  \frac{2\sin(y)}{(\pi+y)(\pi-y)}-\frac{2\sin(y)}{(\pi+2(n_x+i)\pi)(\pi-2(n_x+i)\pi)}dy}\\
	&+ \abs{\int^{2(n_x+i+1)\pi}_{2(n_x+i)\pi}  \frac{2\sin(y)}{(\pi+(n_x+i)\pi)(\pi-(n_x+i)\pi)}dy}\\
	&\leq 4\pi \frac{\pi^2((n_x+i+1)^2-(n_x+i)^2)}{\pi^4(1-(n_x+i)^2)^2}= \frac{4(2n_x+2i+1)}{\pi(1-(n_x+i)^2)^2} \\
	&\le \frac{20}{\pi(n_x+i)^3}  \le \frac{5}{\pi(n_x+i+1)^3} \text{ for } n_x \geq 2.
\end{align*}
Hence
\begin{align*}
	\abs{\sum^\infty_{i=0}  \int^{2(n_x+i+1)\pi}_{2(n_x+i)\pi}  \frac{2\sin(y)}{(\pi+y)(\pi-y)}dy}\leq \frac{5}{\pi}\sum_{i=0}^{+\infty} \frac{1}{(n_x+i +1)^3}\le \frac{5}{2\pi \cdot n_x^2} \le \frac{10\pi}{x^2}
\end{align*}
In total,  $ \abs{ \int_{-1}^1e^{ix z}\sigma(z)dz }  \le  \frac{8\pi}{3x^2} + \frac{10\pi}{x^2} = \frac{38\pi}{3x^2} $ .
\end{proof}

\begin{theorem}\label{theorem:conv_rate_lanczos}
	 For \(x \in (0, 2\pi)\),  \(\mathcal{K}_0 \sim O(K^{-1})\) and \(\mathcal{K}_1 \sim O(K^{-1})\) for the Lanczos filter, which implies \(|F_X^{\sigma}(x) - F_X(x)| \sim O(K^{-1})\) for the Lanczos filter. 
\end{theorem}

\begin{proof}
By applying the Poisson summation formula \ref{eq:poisson_k0}, we obtain
	\begin{align*}
		\mathcal{K}_0(x) 
		=    K \sum_{k=-\infty}^{-2} \hat{s} (Kx + 2kK\pi) 
		+  K \sum_{k=1}^{\infty} \hat{s} (Kx + 2kK\pi)  
		+  K \bigg[   \hat{s} (Kx) +  \hat{s} (Kx - 2K\pi) \bigg],
	\end{align*}
where $\hat{s}$ is defined in Equation \ref{eq:s_hat}.

  For $k\leq -2$ or $k \geq 1$,  we have $ |Kx + 2kK\pi| \geq 2\pi$, which implies  $ |\hat{s}(Kx + 2kK\pi )| \leq \frac{38\pi}{3|Kx + 2kK\pi|^2} $ in light of the Lemma \ref{theorem:conv_rate_transform_of_lanczos}. Hence, both $K \sum_{k=-\infty}^{-2} \hat{s} (Kx + 2kK\pi)$ and $K \sum_{k=1}^{\infty} \hat{s} (Kx + 2kK\pi) $ are bounded by $\frac{19}{6\pi K} \sum_{k=1}^{\infty}  \frac{1}{k^2} $. 
  
  To bound  $K \bigg[   \hat{s} (Kx) +  \hat{s} (Kx - 2K\pi) \bigg]$, note that for $k = -1$ and $k = 0$, $\hat{s} (Kx)$ and $\hat{s} (Kx - 2K\pi)$ are be bounded by $\frac{38\pi}{3K^2x^2}$ and $\frac{38\pi}{3 (Kx-2K\pi)^2}$ respectively, for $K$ large enough such that $\abs{Kx}$ and $\abs{Kx-2K\pi}$ are greater than $2\pi$. 
  
 Thus, for  $K > \max(\frac{2\pi}{x}, \frac{2\pi}{2\pi - x})$,  
  \[
  \abs{\mathcal{K}_0(x)}  \leq \frac{19}{3\pi K} \sum_{k=1}^{\infty}  \frac{1}{k^2}  + \frac{38\pi}{3Kx^2} + \frac{38\pi}{3K (x-2\pi)^2} \leq \frac{38}{3\pi K}+ \frac{38\pi}{3Kx^2} + \frac{38\pi}{3K(x-2\pi)^2}. 
  \]
  Therefore, 
  	\begin{align}\label{ineq:K1_bound_lanczos}
  	\abs{\mathcal{K}_1(x)} &= \abs{\int_{\pi}^x \mathcal{K}_0(t) dt} \nonumber \\
  	& \leq \frac{38}{3K\pi}\abs{x-\pi}  + \frac{38}{3K\pi} \big( \frac{1}{x} +  \frac{1}{2\pi - x} \big) 
  \end{align} 

\end{proof}

\subsection{The raised cosine filter}\label{sec:raised_cosine}
The raised cosine filter is a second order filter ($p=2$). In Theorem \ref{theorem:conv_rate_raised_cosine}  we demonstrate that the filtered Fourier series converges as $O(K^{-2})$ for the raised cosine filter. 

\begin{theorem}\label{theorem:conv_rate_raised_cosine}
For \(x \in (0, 2\pi)\),  \(\mathcal{K}_0 \sim O(K^{-2})\) and \(\mathcal{K}_1 \sim O(K^{-2})\) for the raised cosine filter, which implies \(|F_X^{\sigma}(x) - F_X(x)| \sim O(K^{-2})\). 
\end{theorem}

\begin{proof}
Again, by applying the Poisson summation formula \ref{eq:poisson_k0}, we obtain
\begin{align*}
	\mathcal{K}_0(x) 
	=    K \sum_{k=-\infty}^{-2} \hat{s} (Kx + 2kK\pi) 
	+  K \sum_{k=1}^{\infty} \hat{s} (Kx + 2kK\pi)  
	+  K \bigg[   \hat{s} (Kx) +  \hat{s} (Kx - 2K\pi) \bigg]
\end{align*}
	with $\hat{s}(y) =   \frac{-\pi^2\sin(y)}{y^3 - y \pi^2}$.
	
	
  For $k\leq -2$ or $k \geq 1$,  we have $ |Kx + 2kK\pi| \geq 2\pi$, which implies $ |\hat{s}(Kx + 2kK\pi )| \leq \frac{4\pi^2}{3|Kx + 2kK\pi|^3} $. Hence, both $K \sum_{k=-\infty}^{-2} \hat{s} (Kx + 2kK\pi)$ and $K \sum_{k=1}^{\infty} \hat{s} (Kx + 2kK\pi) $ are bounded by $\frac{1}{6\pi K^2} \sum_{k=1}^{\infty}  \frac{1}{k^3} $.
 We denote  $ \sum_{k=1}^{\infty}\frac{1}{k^3} $ by the Ap\'ery's constant $\zeta(3)$, which is less than 1.21.  
	
%
%
%

 To bound  $K \bigg[   \hat{s} (Kx) +  \hat{s} (Kx - 2K\pi) \bigg]$, note that for $k = -1$ and $k = 0$,  $\hat{s} (Kx)$ and $\hat{s} (Kx - 2K\pi)$ are be bounded by $ \frac{4\pi^2}{3K^3 x^3} $ and $\frac{4\pi^2}{3K^3(2\pi - x )^3}$ respectively, for $K$ large enough such that $\abs{Kx}$ and $\abs{Kx-2K\pi}$ are greater than $2\pi$. 

Thus, for  $K > \max(\frac{2\pi}{x}, \frac{2\pi}{2\pi - x})$,  
\[
\abs{\mathcal{K}_0(x)}  \leq \frac{1}{ 3\pi K^2} \zeta(3)   + \frac{4\pi^2}{3K^2 x^3} + \frac{4\pi^2}{3K^2(2\pi - x )^3} .
\]
Therefore, 
\begin{align}\label{ineq:K1_bound_raised_cosine}
	\abs{\mathcal{K}_1(x)} &= \abs{\int_{\pi}^x \mathcal{K}_0(t) dt} \nonumber \\
	& \leq \frac{1}{ 3\pi K^2} \zeta(3) \abs{x - \pi}    + \frac{2\pi^2}{3K^2} \big(\frac{1}{x^2} + \frac{1}{(2\pi - x)^2} \big). 
\end{align}

\end{proof}

\subsection{The sharpened raised cosine filter}\label{sec:sharpened_raised_cosine}
The sharpened raised cosine filter is an eighth order filter ($p=8$). Using the same approach as for the raised cosine filter, in Theorem \ref{theorem:conv_rate_sharpened_raised_cosine} we show that the convergence rate of the sharpened raised cosine filter is $O(K^{-8})$. That is also one order faster than previously documented in \cite{Vandeven1991} and \cite{GottbliebAndShu1991}.

\begin{theorem}\label{theorem:conv_rate_sharpened_raised_cosine}
 For \(x \in (0, 2\pi)\), \(\mathcal{K}_0 \sim O(K^{-8})\) and  \(\mathcal{K}_1 \sim O(K^{-8})\) , which implies \(|F_X^{\sigma}(x) - F_X(x)| \sim O(K^{-8})\). 
\end{theorem}

\begin{proof}
In the case of the sharpened raised cosine filter,
	\[
	\hat{s}(y) =  \int_{-1}^{1} \sigma(x) e^{iyx}dx =  \frac{ 11025 \pi^8 \sin(y)}{y^9 - 84 \pi^2 y^7 + 1974 \pi^4 y^5 - 12916 \pi^6 y^3 + 11025 \pi^8 y}.
	\]
	
Note that
\begin{align*}
&y^9 - 84 \pi^2 y^7 + 1974 \pi^4 y^5 - 12916 \pi^6 y^3 + 11025 \pi^8 y \\
&=  y [z^2(\frac{21}{22}z - 44\pi^2)^2 + 38\pi^4 (z - \frac{3229}{19}\pi^2)^2 + \frac{1}{484} z^4 + \frac{21}{242} z^4 + \pi^8 (11025 - 2\frac{3229^2}{19})],
\end{align*}

with $z = y^2$.
\\
If $\abs{y} \geq 6\pi$, we have $\frac{21}{242} z^4 + \pi^8 (11025 - 2\frac{3229^2}{19}) \geq  0$, which implies that
 \[ \abs{y^9 - 84 \pi^2 y^7 + 1974 \pi^4 y^5 - 12916 \pi^6 y^3 + 11025 \pi^8 y } \leq \abs{\frac{1}{484} y^9}, \]
 and \[
 \abs{\hat{s}(y)} \leq 5336100  \pi^8 y^{-9}.
 \]

Without any loss we could assume that $K \geq 3$.  For $k\leq -2$ or $k \geq 1$,  we have $ |Kx + 2kK\pi| \geq 6\pi$, which implies $ |\hat{s}(Kx + 2kK\pi )| \leq  5336100  \pi^8 (Kx + 2kK\pi)^{-9}  $. Hence, both $K \sum_{k=-\infty}^{-2} \hat{s} (Kx + 2kK\pi)$ and $K \sum_{k=1}^{\infty} \hat{s} (Kx + 2kK\pi) $ are bounded by $\frac{1334025}{\pi 2^7 K^8} \sum_{k=1}^{\infty}  \frac{1}{k^9} $.
We denote  $ \sum_{k=1}^{\infty}\frac{1}{k^9} $ by  $\zeta(9)$.

 To bound  $K \bigg[   \hat{s} (Kx) +  \hat{s} (Kx - 2K\pi) \bigg]$, note that for $k = -1$ and $k = 0$,  $\hat{s} (Kx)$ and $\hat{s} (Kx - 2K\pi)$ are be bounded by $ 5336100  \pi^8 (Kx)^{-9}$ and $ 5336100  \pi^8 (K (2\pi - x))^{-9}$ respectively, for $K$ large enough such that $\abs{Kx}$ and $\abs{Kx-2K\pi}$ are greater than $6\pi$. 
 
  Finally, for  $K > \max(\frac{6\pi}{x}, \frac{6\pi}{2\pi - x})$,  
 \[
 \abs{\mathcal{K}_0(x)}  \leq \frac{1334025}{\pi 2^7 K^8} \zeta(9)   + \frac{5336100 \pi^8}{K^8x^9} + \frac{5336100 \pi^8}{K^8 (x-2\pi)^9} .
 \]
 Finally, 
 \begin{align}\label{ineq:K1_bound_sharpened_raised_cosine}
 	\abs{\mathcal{K}_1(x)} &= \abs{\int_{\pi}^x \mathcal{K}_0(t) dt} \nonumber \\
 	& \leq \frac{1334025}{\pi 2^7 K^8} \zeta(9) \abs{x - \pi}    + \frac{5336100 \pi^8}{8K^8} \big(\frac{1}{x^8} + \frac{1}{(2\pi - x)^8} \big). 
 \end{align}
 
%
\end{proof}

\subsection{The second order exponential filter}\label{sec:exponential_filter}
In the case of the second order exponential filter $\sigma(x) = e^{-\alpha x^2}$, 
\[
\hat{s}(x) = 2 \int_0^1 e^{-\alpha t^2}\cos(tx)dt =  \frac{2}{x}e^{-\alpha}\sin(x)  -\frac{4\alpha e^{-\alpha} \cos(x)}{x^2} + \frac{4\alpha}{x^2} \int_0^1 \cos(tx) e^{-\alpha t^2}(1 - 2\alpha t^2 )dt.
\]

\begin{lemma}\label{lemma:third_order_term_exp}
 For $\alpha > 0$ and $|x|\leq 2\pi $, we have the following inequality  
	\begin{align*}
		\abs{\int_0^1 \cos(tx) e^{-\alpha t^2}(1 - 2\alpha t^2 )dt} \leq \frac{1 + 4e^{-1}}{\abs{x}}.
	\end{align*} 
\end{lemma}

\begin{proof}
Note that integration by parts yields
\begin{align*}
\int_0^1 \cos(tx) e^{-\alpha t^2} dt 
= \frac{1}{x} e^{-\alpha} \sin(x)- \frac{1}{x} \int_0^1 \sin(tx) d[e^{-\alpha t^2} ].
\end{align*}
Since $e^{-\alpha t^2}$ is a monotonically decreasing function, 
\[
\abs{\frac{1}{x}\int_0^1 \sin(tx) d[e^{-\alpha t^2} ]} \leq \abs{\frac{1}{x}}\cdot \abs{\int_0^1 d[e^{-\alpha t^2} ]} = \frac{1 - e^{-\alpha}}{\abs{x}}.
\]
Therefore,
\[
\abs{\int_0^1 \cos(tx) e^{-\alpha t^2} dt} \leq  \frac{e^{-\alpha} }{\abs{x}}  + \frac{1 - e^{-\alpha}}{\abs{x}}  = \frac{1}{\abs{x}},
\]
Similarly, integration by parts yields
\begin{align*}
	\int_0^1 \cos(tx) t^2 e^{-\alpha t^2} dt 
	= \frac{1}{x}e^{-\alpha} \sin(x) - \frac{1}{x} \int_0^1 \sin(tx) d[t^2 e^{- \alpha t^2} ].
\end{align*}
Further,
\begin{align*}
 \abs{\int_0^1 \sin(tx) d[t^2 e^{- \alpha t^2} ]} &\leq \abs{\int_0^{1/\sqrt{\alpha}} \sin(tx) d[t^2 e^{- \alpha t^2} ] } +  \abs{\int_{\sqrt{\alpha}}^1 \sin(tx) d[t^2 e^{- \alpha t^2} ] } \\
 & \leq \abs{\int_0^{1/\sqrt{\alpha}}  d[t^2 e^{- \alpha t^2} ] } +  \abs{\int_{\sqrt{\alpha}}^1  d[t^2 e^{- \alpha t^2} ] } \\
 &=  (2 \frac{e^{-1}}{a} - e^{-a})
\end{align*}
where the second inequality is due to that $t^2 e^{- \alpha t^2}$ is monotonic on $[0,1/\sqrt{\alpha}  ]$ and $[1/\sqrt{\alpha}, 1]$. 

Therefore, 
\[
\abs{\int_0^1 \cos(tx) t^2 e^{-\alpha t^2} dt} \leq \frac{e^{-\alpha} }{\abs{x}} + \frac{1}{\abs{x}} (2 \frac{e^{-1}}{a} - e^{-a}) = \frac{2e^{-1}}{a\abs{x}},
\]

In total,
\[
\abs{\int_0^1 \cos(tx) e^{-\alpha t^2}(1 - 2\alpha t^2 )dt} \leq \abs{\int_0^1 \cos(tx) e^{-\alpha t^2} dt} + 2\alpha \abs{\int_0^1 \cos(tx) t^2 e^{-\alpha t^2} dt} =\frac{1 + 4e^{-1}}{\abs{x}} 
\]

\end{proof}

\begin{theorem}\label{theorem:conv_rate_second_order_exp}
	For  \(x \in (0, 2\pi)\) and $\alpha>0$, the bounds of $\mathcal{K}_0$ and $\mathcal{K}_1$ with the second order exponential filter depends are as follows:
	\begin{align*}
	\abs{\mathcal{K}_0(x)} &\leq 
	( 4 \abs{\pi - x} e^{-\alpha} ) \left( \frac{1}{24} + \frac{1}{x(2\pi-x)} \right) + \nonumber \\
	&\frac{4\alpha e^{-\alpha}}{K} \left( \frac{1}{x^2} + \frac{1}{(x-2\pi)^2} + \frac{1}{12}\right) + \frac{1}{K^2} \left( \frac{20\alpha}{x^3} + \frac{20\alpha}{(2\pi - x)^3} + \frac{6.5\alpha}{\pi^3} \right),
	\end{align*}   
	and 
	\begin{align*}
		\abs{\mathcal{K}_1(x)} & \leq \frac{e^{-\alpha}}{12} (x-\pi)^2 +2 e^{-\alpha} (2\ln(\pi) - \ln(x) -\ln(2\pi -x)) \\
		&  + \frac{4\alpha e^{-\alpha}}{K} \left( \abs{\frac{1}{x} -\frac{1}{\pi} } + \abs{\frac{1}{x-2\pi} + \frac{1}{\pi}} + \frac{1}{12}\abs{x-\pi} \right) \nonumber \\
		&+ \frac{1}{K^2} \left( \abs{\frac{10\alpha}{x^2} -\frac{10\alpha}{\pi^2} } + \abs{\frac{10\alpha}{(2\pi -x)^2} -\frac{10\alpha}{\pi^2} }  + \frac{6.5\alpha}{\pi^3} \abs{x-\pi}\right)
	\end{align*}

\end{theorem}

\begin{proof}
As a result of  Lemma \ref{lemma:third_order_term_exp},
\begin{align*}
	\abs{ \sum_{k=-\infty}^{+\infty} \hat{s}(Kx + 2kK\pi) } &\leq \abs{ \sum_{k=-\infty}^{+\infty} \frac{2 e^{-\alpha} \sin(Kx + 2kK\pi)}{Kx + 2kK\pi} } + \abs{ \sum_{k=-\infty}^{+\infty} \frac{4\alpha e^{-\alpha} \cos(Kx + 2kK\pi)}{(Kx + 2kK\pi)^2} } \\
	&+ \abs{ 20\alpha \sum_{k=-\infty}^{+\infty} \frac{1}{(Kx + 2kK\pi)^3} }.
\end{align*}
	
Regarding the leading term, we have the following bound:
	\begin{align*}
	\abs{ \sum_{k=-\infty}^{+\infty} \frac{2 e^{-\alpha} \sin(Kx + 2kK\pi)}{Kx + 2kK\pi} } & \leq \frac{2 e^{-\alpha}}{K} \abs{ \sum_{k=-\infty}^{+\infty} \frac{1}{x + 2k\pi} }\\
	& = \frac{2 e^{-\alpha}}{K} \abs{ \sum_{k=0}^{+\infty} \frac{1}{x + 2k\pi} - \sum_{k=-\infty}^{-1} \frac{1}{ (2\pi - x) - 2 (k+1)\pi} } \\
	& = \frac{2 e^{-\alpha}}{K} \abs{ \sum_{k=0}^{+\infty} \frac{1}{x + 2k\pi} - \sum_{k=0}^{-\infty} \frac{1}{ (2\pi - x) + 2 k \pi} } \\
	& = \frac{2 e^{-\alpha}}{K} \abs{\sum_{k=0}^{+\infty} \left( \frac{1}{x + 2k\pi} - \frac{1}{ (2\pi - x) + 2 k \pi}  \right) } \\
	&=  \frac{4 \abs{\pi - x} e^{-\alpha}}{K} \sum_{k=0}^{+\infty} \frac{1}{(x + 2k\pi)((2\pi - x) + 2 k \pi)} \\
	&\leq \frac{4 \abs{\pi - x} e^{-\alpha}}{K}  \left( \sum_{k=1}^{+\infty} \frac{1}{4\pi^2 k^2} + \frac{1}{x(2\pi-x)} \right) \\
	& = \frac{4 \abs{\pi - x} e^{-\alpha}}{K} \left( \frac{1}{24} + \frac{1}{x(2\pi-x)} \right)
\end{align*}

The second and third terms can be easily bounded as follows:

\begin{align*}
	\abs{ \sum_{k=-\infty}^{+\infty} \frac{4\alpha e^{-\alpha} \cos(Kx + 2kK\pi)}{(Kx + 2kK\pi)^2} } &\leq \frac{4\alpha e^{-\alpha}}{K^2} \abs{ \sum_{k=-\infty}^{+\infty} \frac{1}{(x + 2k\pi)^2} }\\
 &\leq \frac{4\alpha e^{-\alpha}}{K^2} \left( \frac{1}{x^2} + \frac{1}{(x-2\pi)^2} + \frac{1}{12}\right),
\end{align*}

and 
\begin{align*}
	\abs{ \sum_{k=-\infty}^{+\infty}  \frac{20\alpha }{(Kx + 2kK\pi)^3}} \leq \frac{1}{K^3} \left( \frac{20\alpha}{x^3} + \frac{20\alpha}{(2\pi - x)^3} + \frac{6.5\alpha}{\pi^3} \right).   
\end{align*}

Thus 
\begin{align*}
\abs{\mathcal{K}_0(x)} &\leq 
 ( 4 \abs{\pi - x} e^{-\alpha} ) \left( \frac{1}{24} + \frac{1}{x(2\pi-x)} \right) + \nonumber \\
  &\frac{4\alpha e^{-\alpha}}{K} \left( \frac{1}{x^2} + \frac{1}{(x-2\pi)^2} + \frac{1}{12}\right) + \frac{1}{K^2} \left( \frac{20\alpha}{x^3} + \frac{20\alpha}{(2\pi - x)^3} + \frac{6.5\alpha}{\pi^3} \right),
\end{align*}
and

 \begin{align}\label{ineq:K1_bound_exp2}
	&\abs{\mathcal{K}_1(x)} = \abs{\int_{\pi}^x \mathcal{K}_0(t) dt} \nonumber \\
	& \leq \frac{e^{-\alpha}}{12} (x-\pi)^2 +2 e^{-\alpha} (2\ln(\pi) - \ln(x) -\ln(2\pi -x)) + \frac{4\alpha e^{-\alpha}}{K} \left( \abs{\frac{1}{x} -\frac{1}{\pi} } + \abs{\frac{1}{x-2\pi} + \frac{1}{\pi}} + \frac{1}{12}\abs{x-\pi} \right) \nonumber \\
	&+ \frac{1}{K^2} \left( \abs{\frac{10\alpha}{x^2} -\frac{10\alpha}{\pi^2} } + \abs{\frac{10\alpha}{(2\pi -x)^2} -\frac{10\alpha}{\pi^2} }  + \frac{6.5\alpha}{\pi^3} \abs{x-\pi}\right)
\end{align}

\end{proof}

\begin{remark}\label{remark:exp2}
\cite{GottbliebAndShu1991} suggests to choose $\alpha$ such that $e^{-\alpha}$ falls within the roundoff error of the specific computer. When $e^{-\alpha} \approx 0$, we observe that $\abs{\mathcal{K}_1(x)}$ is approximately an order of  $ O(K^{-2})$, which implies $|F_X^{\sigma}(x) - F_X(x)|$ is approximately an order of  $ O(K^{-2})$ as well.
\end{remark}

\begin{remark}\label{remark:exp2_parametrized}
Alternatively, we could parameterize $\alpha$, e.g.,  $\alpha = -\ln(1/K^2)$,  to ensure that $\mathcal{K}_1$ with the second order exponential filter attains the convergence rate of $O(K^{-2})$. 
\end{remark}

\section{Convergence of the implied moments}\label{sec:cos_moment}
The COS-recovered CDF also enables us to efficiently calculate the corresponding moments with a semi-analytical series expansion. E.g., the COS-implied first moment is given by
\begin{equation}
\int_a^b x d F_X^{\sigma}(x) = \frac{A_0}{4}(b^2 - a^2) + A_k \sigma(k/K) \left(\frac{ b - a }{k\pi} \right)^2 \sum_{k=1}^{K} (\cos{k\pi} - 1)
\end{equation}

And the COS-implied second moment is given by
\begin{equation}
\int_a^b x^2 d F_X^{\sigma}(x) = \frac{A_0}{6}(b^3 - a^3) + 2 A_k \sigma(k/K) \left(\frac{ b - a }{k\pi} \right)^2 \sum_{k=1}^{K} (b\cos{k\pi} - a)
\end{equation}

In the general case, 

\begin{align}
	\int_a^b x^q d F_X^{\sigma}(x) &= \int_a^b x^q  \left(\frac{A_0}{2}+ \sum_{k=1}^{K} A_k \sigma(k/K) \cos\left(k\pi \frac{x-a}{b-a}\right) \right) d x \nonumber \\
	&= \frac{A_0}{2(q+1)}(b^{q+1} - a^{q+1}) + \sum_{k=1}^K \sigma(k/K)A_k C_k,
\end{align}
where $C_k$ is the Fourier cosine transform of $x^q$, i.e., $C_k =\int_a^b x^q  \cos\left(k\pi \frac{x-a}{b-a}\right)  $.

In the following theorem, we demonstrate the convergence of the COS implied moments and analyze the convergence rate.

\begin{theorem}\label{thm_moment}
	Consider a discrete random variable $X$ with finite number of possible values $\{\mathcal{X}_{0\leq m \leq M}\}$. Without loss of generality, we assume $0 < \mathcal{X}_0 < \mathcal{X}_1 < \cdots < \mathcal{X}_M < \pi$. Let $q$ be a positive integer equal to or less than $p$. The COS implied moment, i.e., $\int_0^{\pi} x^q dF_{X}^{\sigma}$,  converges to the actual moment   $\int_0^{\pi} x^q dF_{X}$ for any positive integer number $q$. The convergence error can be bounded as follows:
	\begin{equation}
	\abs{\int_0^{\pi} x^q dF_{X}^{\sigma}- \int_0^{\pi} x^q d dF_{X} } \leq  C_1 O(K^{1-p}) +  C_2 O(K^{1/2-q})
	\end{equation}
	for some positive constants $C_1$ and $C_2$. 
\end{theorem}

\begin{proof}
	Integration by parts yields
	\begin{align*}
	\int_0^{\pi} x^q dF_{X}^{\sigma}- \int_0^{\pi} x^q d dF_{X}   = {\pi}^q ( F_X^{\sigma}(\pi) -  F_X(\pi) ) - q \big( \int_0^{\pi}x^{q-1} (F_{X}^{\sigma}(x) - F_{X}(x)) dx \big) 
	\end{align*}

Note that
\begin{align*}
&q\int_0^{\pi}x^{q-1} (F_{X}^{\sigma}(x) - F_{X}(x)) dx = q\int_0^{\pi} x^{q-1} \big(   \sum_{m=1}^{M-1} \frac{ p_m}{2\pi} \left( \mathcal{K}_1(x+\mathcal{X}_m) +  \mathcal{K}_1(x- \mathcal{X}_m + 2\pi) \right) \mathbf{1}_{0 < x < \mathcal{X}_m} \nonumber \\
&+ \sum_{m=1}^{M-1}  \frac{p_m }{2\pi} \left( \mathcal{K}_1(x-\mathcal{X}_m) +  \mathcal{K}_1(x + \mathcal{X}_m) \right) \mathbf{1}_{\mathcal{X}_m < x < \pi} \big)dx \\
&= \sum_{m=1}^{M-1} \frac{ p_m}{2\pi} q\int_0^{\pi} x^{q-1}\mathcal{K}_1(x+\mathcal{X}_m) \mathbf{1}_{0 < x < \mathcal{X}_m}dx + \sum_{m=1}^{M-1} \frac{ p_m}{2\pi} q\int_0^{\pi} x^{q-1}\mathcal{K}_1(x - \mathcal{X}_m + 2\pi) \mathbf{1}_{0 < x < \mathcal{X}_m}dx \\
&+ \sum_{m=1}^{M-1} \frac{ p_m}{2\pi} q\int_0^{\pi} x^{q-1}\mathcal{K}_1(x - \mathcal{X}_m) \mathbf{1}_{\mathcal{X}_m < x < \pi}dx + \sum_{m=1}^{M-1} \frac{ p_m}{2\pi} q\int_0^{\pi} x^{q-1}\mathcal{K}_1(x + \mathcal{X}_m) \mathbf{1}_{\mathcal{X}_m < x <\pi}dx
\end{align*}

Repeating integration by parts allows to estimate a bound for each of the terms above. E.g.,  $\abs{q\int_0^{\pi} x^{q-1} K_1(x + \mathcal{X}_m) \mathbf{1}_{\mathcal{X}_m < x < \pi} dx} < \abs{q\int_0^{\pi} x^{q-1} K_1(x + \mathcal{X}_m) dx}$, and,
\begin{align*}
q\int_0^{\pi} x^{q-1} K_1(x + \mathcal{X}_m) dx &= \sum_{j=2}^q (-1)^j \prod_{l=1}^{j-1} {(q+1-l)} \cdot \pi^{q+1-j} K_j(\pi + \mathcal{X}_m) + (-1)^{q+1} q! \int_0^{\pi} K_q(x+\mathcal{X}_m)dx.
\end{align*}

On the one hand,  $F_X^{\sigma}(\pi) -  F_X(\pi)$ and $K^q(\cdot)$ converge with the rate of $O(K^{1-p})$.

On the other hand, the Cauchy-Schwarz inequality and the periodicity of $K_q$ implies that
\[
\abs{\int_0^{\pi} K_q(x+\mathcal{X}_m)dx} \leq \sqrt{\pi} \sqrt{\int_0^{\pi} K_q^2(x+\mathcal{X}_m)dx }
\]

Note that $\int_0^{2\pi} K_q^2(x)dx \sim O(K^{1-2q})$ (see e.g., Proposition 4 in \cite{Vandeven1991}), we obtain
\[
\int_0^{\pi} K_q^2(x+\mathcal{X}_m)dx \leq \int_0^{2\pi} K_q^2(x)dx \sim O(K^{1-2q})
\]

Therefore in total, 
\[
\abs{\int_0^{\pi} x^q dF_{X}^{\sigma}- \int_0^{\pi} x^q d dF_{X} } \leq  C_1 O(K^{1-p}) +  C_2 O(K^{1/2-q})
\]
for some positive constants $C_1$ and $C_2$. 
\end{proof}

\section{Numerical examples and applications}\label{sec:numerical_example}
We first perform numerical tests to verify the validity of the derived bounds for $\mathcal{K}_1$. Next, we explore two applications of the discrete COS method: solving the distributions of the generalized Poisson-binomial distributions and the affine Hawkes processes.

\subsection{Numerical Tests of the Bounds for $\mathrm{K1}$}
The bounds of $\mathcal{K}_1$ in Inequality \ref{ineq:K1_bound_lanczos}, \ref{ineq:K1_bound_raised_cosine}, \ref{ineq:K1_bound_sharpened_raised_cosine}, and \ref{ineq:K1_bound_exp2} are numerically verified over a grid of 1,000 evenly spaced points between $0$ and $2\pi$, excluding the endpoints $0$ and $2\pi$. The testing codes are available at the Github repository.

To illustrate the convergence behavior of $K1$, we plot the values of $\mathcal{K}_1(0.5)$ w.r.p an increasing number of Fourier expansion terms. See Figure \ref{fig:bound_lanczos}, \ref{fig:bound_raised_cosine}, \ref{fig:bound_sharpened_raised_cosine}, and  \ref{fig:bound_exp2}, where the bounds are calculated using our formulas \ref{ineq:K1_bound_lanczos}, \ref{ineq:K1_bound_raised_cosine}, \ref{ineq:K1_bound_sharpened_raised_cosine}, and \ref{ineq:K1_bound_exp2}. For the second-order exponential filter, we set $\alpha=16$ to suppress the contributions of the terms on the right hand side of \ref{ineq:K1_bound_exp2} that involve the factor $e^{-\alpha}$.

Alternatively, we parameterize $\alpha = -\ln(1/K^2)$ to ensure that $\mathcal{K}_1$ with the second order exponential filter achieves the convergence rate of $O(K^{-2})$. See the convergence behavior in Figure \ref{fig:bound_exp2_parametrized}.

%
%
%

\begin{table}[h!]
	\centering
	\begin{tabular}{cc}
		\begin{minipage}{0.45\textwidth}
			\centering
			\includegraphics[width=\linewidth]{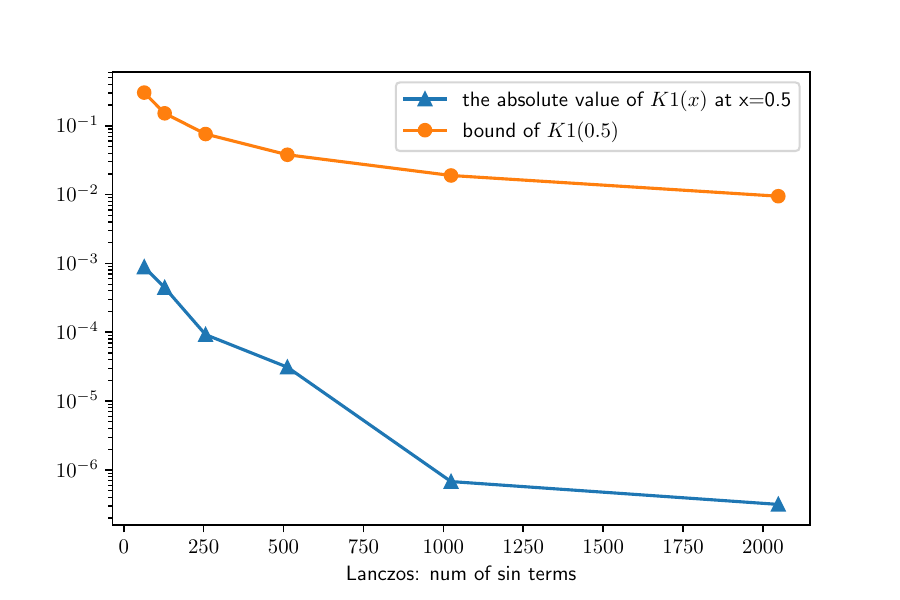}
			\captionof{figure}{\small Lanczos}\label{fig:bound_lanczos}
		\end{minipage} &
		\begin{minipage}{0.45\textwidth}
			\centering
			\includegraphics[width=\linewidth]{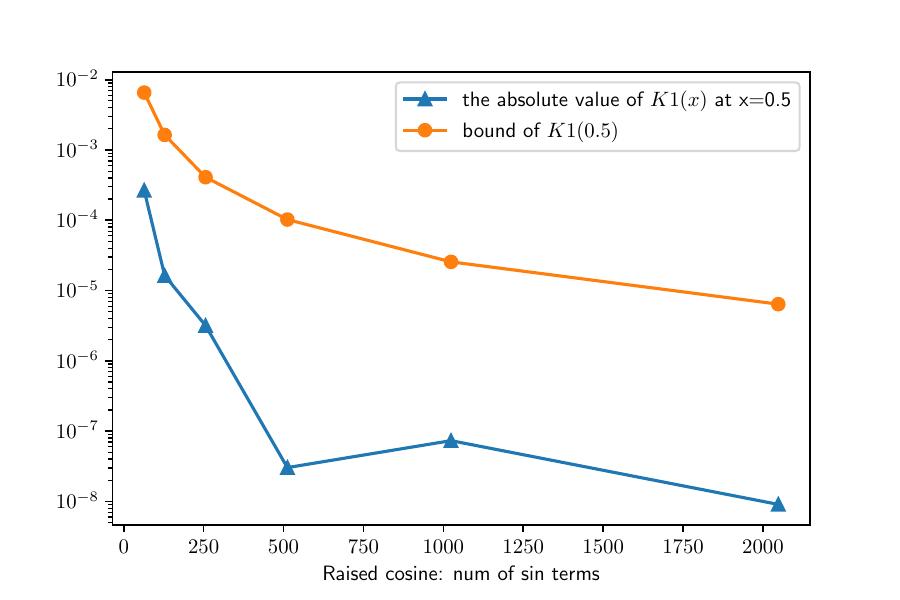}
			\captionof{figure}{\small Raised cosine}\label{fig:bound_raised_cosine}
		\end{minipage} \\
		
		\begin{minipage}{0.45\textwidth}
			\centering
			\includegraphics[width=\linewidth]{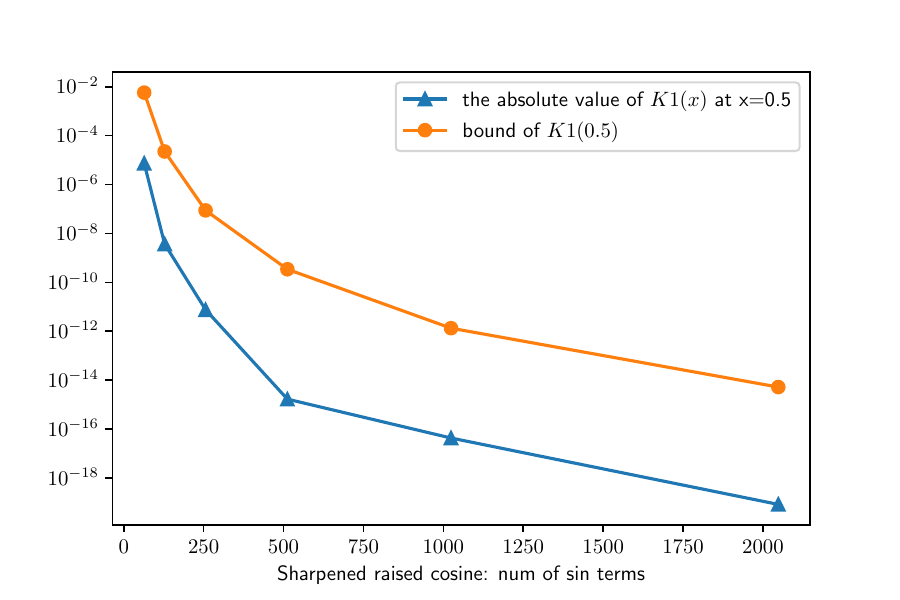}
			\captionof{figure}{\small Sharpened raised cosine}\label{fig:bound_sharpened_raised_cosine}
		\end{minipage} &
		\begin{minipage}{0.45\textwidth}
			\centering
			\includegraphics[width=\linewidth]{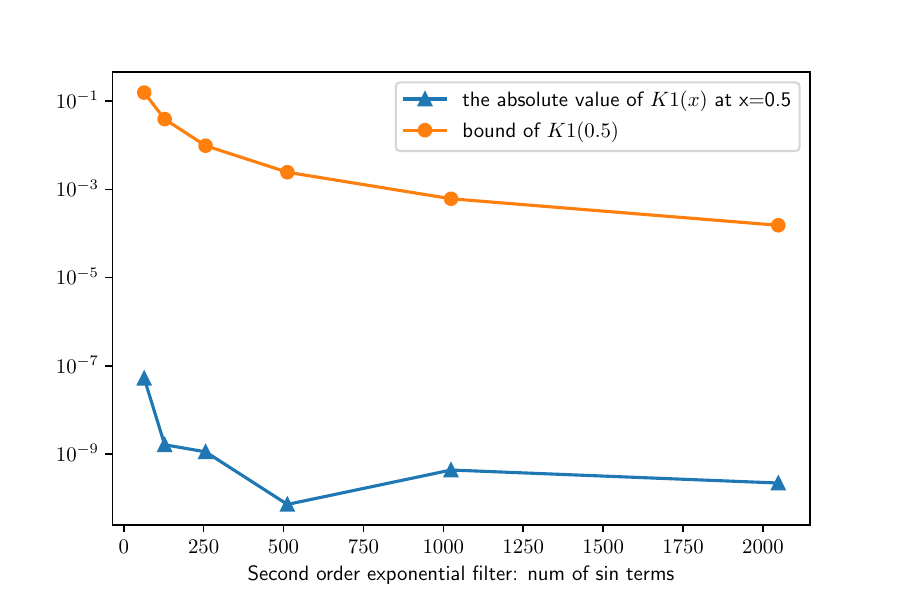}
			\captionof{figure}{\small Second order exponential}\label{fig:bound_exp2}
		\end{minipage}
	\end{tabular}
\end{table}

\begin{figure}[!htbp]
	\begin{center}
		\includegraphics[scale=0.66]{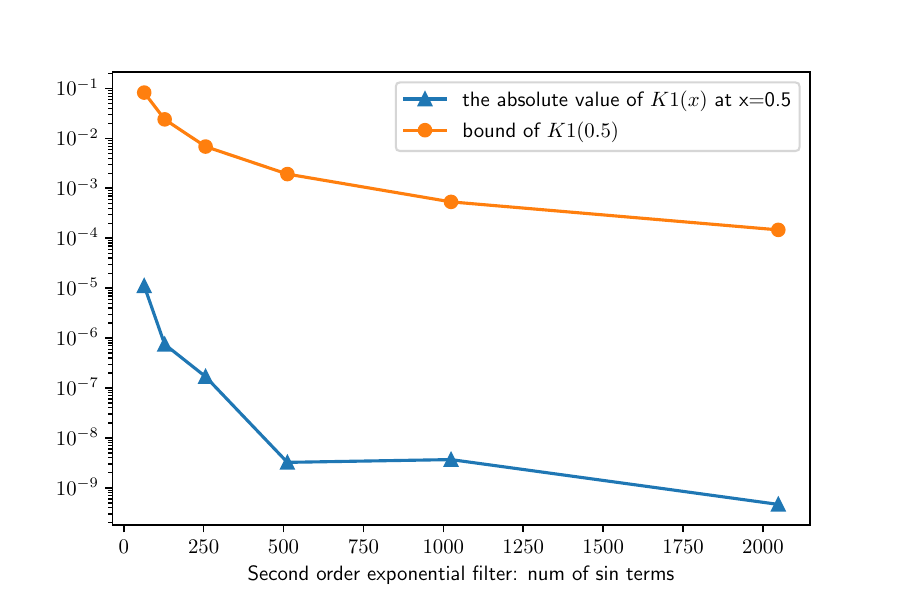}
		\caption{Bound of $\mathrm{K}_1$ with the second order exponential filter. $\alpha =-\ln(1/K^2) $}\label{fig:bound_exp2_parametrized}
	\end{center}
\end{figure}

\subsection{Generalized Poisson-binomial distribution}
The first application we consider involves solving the CDF of PB and GPB distributions. The PB distribution is the distribution of the sum of independent, non-identically distributed Bernoulli random variables. The GPB distribution extends this by replacing the Bernoulli variables in the PB distribution with two-point random variables that can take arbitrary values instead of just 0 and 1. A GPB random variable $X$ can be represented as

\begin{equation}\label{eq:GPB}
	X = \sum_{n=1}^N (a_n (1 - I_n) + b_n I_n),
\end{equation} 
where $\{I_{n}\}_{1\leq n \leq N}$ is a sequence of independent and non-identically distributed Bernoulli random variables with $\mathrm{P}(I_n=1) = p_n$. If $a_n=0$ and $b_n=1$ for all $n$, $X$ reduces to a PB random variable.

 The PB and GPB distributions have broad applications in fields such as statistics, actuarial science, and voting theory. While the computation of their CDFs has been extensively studied, existing methods become computationally expensive for large values of $N$. Recursive methods, such as those in \cite{ChenAndLiu1997}, exhibit a worst-case computational complexity of $O(N^2)$
, whereas FFT-based approaches (e.g., \cite{BZB2018}, \cite{FernandezAndWilliams2010}, \cite{Hong2013}, \cite{ZHB2018}) still has a complexity of $O(N\ln{N})$.

In contrast, the COS method offers a semi-analytical formula for easily approximating the CDF. Note that the ch.f. of $X$ is given by
\[
\varphi_X(\omega) = \prod_{n=1}^N \left((1-p_n)e^{i\omega a_n} + p_n e^{i\omega b_n}  \right),
\]
which can be inverted in a single-step calculation (Equation \ref{eq:filtered_cos_cdf} ) to yield the CDF.

We first apply the semi-analytical COS formula to a two-point distribution, which is a special case of the PB and GPB distributions. Here,  $X$ takes values on $\pi/4$ and $\pi/2$ with $\mathrm{P}(X=\pi/4) = 0.4$. Table \ref{tab:cos_bernoulli_error}  presents the approximation error of COS when using the raised cosine filter to evaluate the CDF at $0.6\pi$. The convergence behavior is plotted in Figure \ref{fig:TwoPoint}.

\begin{table}[H]
	\centering
	\begin{tabular}{|c|c|c|c|c|c|}
		\hline
		Number of Fourier terms  & 16 & 32 & 64 & 128 & 256\\
		\hline
		Absolute error  & 3.3e-3 & 7.8e-4 & 4.7e-5 & 8.6e-6 & 3.7e-7      \\
		\hline
	\end{tabular}
	\caption{Error convergence of the COS CDF of a two-point distribution at $0.6\pi$ with the raised cosine filter. }\label{tab:cos_bernoulli_error}
\end{table}

\begin{figure}[!htbp]
	\begin{center}
		\includegraphics[width=1.0\textwidth]{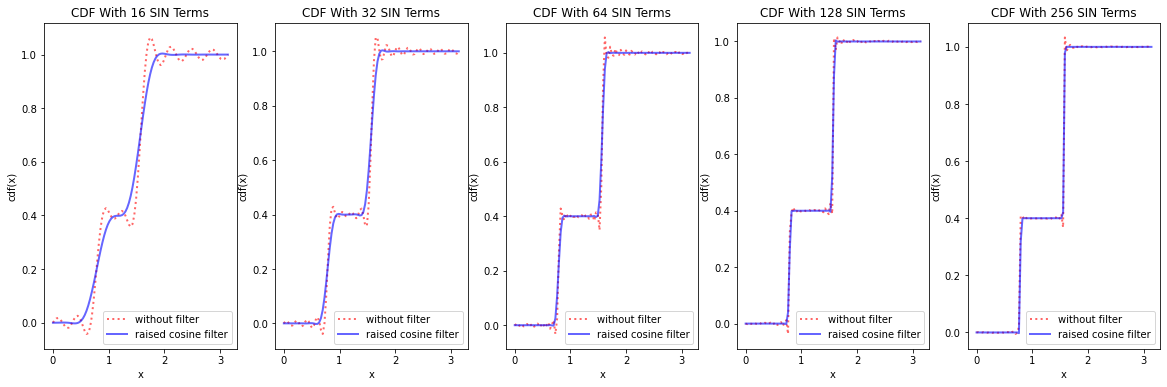}
		\caption{COS CDF of a two-point distribution with the raised cosine filter.}\label{fig:TwoPoint}
	\end{center}
\end{figure}

Next, we consider a GPB random variable $X$ by setting $N=95$ in Equation \ref{eq:GPB}, with
\begin{itemize}
	\item The sequence of Bernoulli random variable $I_n$ has probabilities $p_n$ of $1\%, 2\%, 3\%, \cdots$ up to $95\%$ of taking the value 1;
	\item Each $b_n$ is independently sampled from the uniform distribution on $[0, 1]$;
	\item Each $a_n$ is set to be half of $b_n$.
\end{itemize}

Figure \ref{fig:GPB} shows the COS-recovered CDF of $X$ using the raised cosine filter and 128 Fourier expansion terms, and the CDF estimated with 1 million Monte Carlo simulation. Note that the two CDFs match well with each other, while the COS CDF can cover the entire support of $X$. The computational time of the COS method is about 1 second on an Intel(R) Core(TM) i5-10210U CPU with 4 cores and base frequency of 1.60 GHz.

\begin{figure}[!htbp]
	\begin{center}
		\includegraphics[width=1.0\textwidth]{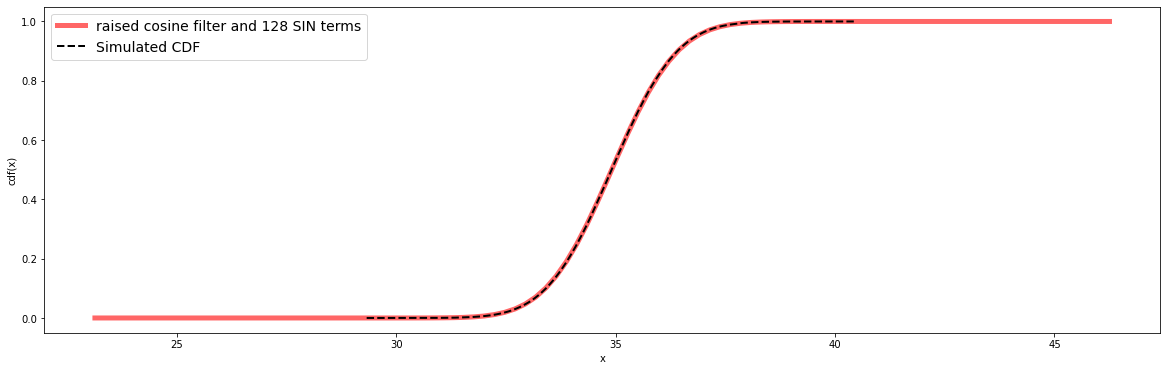}
		\caption{Comparison between the COS CDF and Monte Carlo simulated CDF}\label{fig:GPB}
	\end{center}
\end{figure}

\subsection{Affine Hawkes processes}
Another application we explore is the affine Hawkes process proposed in \cite{Errais2020} for modeling portfolio credit risk. Let $T_{n, n\geq 1}$ be an increasing sequence of stopping times   defined on the probability space $(\Omega, \mathcal{F}, P)$ with right continuous and complete information filtration $\mathbb{F} = (\mathcal{F}_t)_{t\geq 0}$. In this framework, the defaults in a credit-risky portfolio are characterized by the sequence of random variables $(T_n, l_n)$, where $T_n$  represents the time of the $n$th default, and $l_n$ denotes the corresponding random loss.  The number of defaults by time $t$, $N_t$, is modeled as a point process: $N_t = \sum_{n \geq 1} \boldmath{1}_{(T_n \leq t)}$. The associated loss process is given by $L_t = \sum_{n \geq 1} l_n \boldmath{1}_{(T_n \leq t)}$. \\

Further, $N_t$ is assumed to follow a Hawkes process, with the intensity process $\lambda_t$  satisfying the stochastic differential equation:
\[
d\lambda_t = \kappa(c -\lambda_t)dt + \delta d L_t.
\] 

The conditional Fourier transform of $J_t=\left(L_t,N_t\right)$ is provided in \cite{Errais2020}. 

\begin{prop}\label{prop:Hawkes}
	Let $u\in \mathbb{C}^2_{-}$ denote a vector of two complex numbers with non-positive real part, the Fourier transform of $J_T$ conditional on $\mathcal{F}_{t}$ ($t\leq T$) is given by
	\begin{equation}
		\mathrm{E}\left(e^{u\cdot J_T}|\mathcal{F}_t\right)=\exp\left(a(t)+ b(t) \lambda_t + u\cdot J_t\right),
	\end{equation}
	where $a$ and $b$ satisfy the following ordinary differential equations (ODEs):
	\begin{align*}
		\partial_t b(t) &=\kappa b(t) + 1 - \theta\left(\delta b(t) + u\cdot \left(1,0\right)^\intercal \right)e^{u\cdot \left(0,1\right)^\intercal} , \\
		\partial_ta\left(t\right) &=-\kappa c b\left(t\right) ,
	\end{align*}
with the boundary conditions $ b\left(T\right)=0$ and $a\left(T\right)=0$, and $\theta$ denoting the Fourier transform of the loss distribution $\mu$:
	\[
	\theta(\omega)=\int e^{wz}dv(z).
	\]
\end{prop}

There is no closed form expression for the distribution of the number of defaults in a given time interval, i.e.,  $N_T - N_t$ conditional on $\mathcal{F}_t$. Only the first order moment is found in \cite{Errais2020} .  \cite{Errais2020} relies on the Fa\`a di Bruno's formula to evaluate $\mathrm{P}(N_T - N_t| \mathcal{F}_t)$, which is quite cumbersome. However, using the discrete COS method, we can easily get a semi-analytical expression of the CDF of  $N_T$ conditional on $\mathcal{F}_t$.

As a numerical example, we set $t=1$,  $T=2$, $\lambda_t=1$, and $J_t=\left(0.7, 3\right)$, with parameters $\kappa=1.2$, $\delta=0.7$, and $c=1$. The loss at default $l_n$ is an i.i.d. sequence of exponential random variables with the rate parameter $\lambda = 5/6$. Since $N_T$ takes values on $[N_t, +\infty)$, we need define a large enough truncation range, the same as the COS method for the continuous distributions (see \cite{FangAndOosterleeCOS}). The truncation range of $N_T$ is determined using the rule $[N_t, \mathrm{E}_t[N_T] + 25 \cdot \sigma_t(N_T)] $, where $\mathrm{E}_t[N_T]$ and $\sigma_t(N_T)$ are the conditional expectation and standard derivation, approximated by the numerical derivatives of the conditional ch.f of $N_T$. To reduce the impact of Gibbs phenomena on the left side of the truncation range, the lower bound is further shifted downward to $N_t - 0.1 u$, where $u$ refers to the upper bound $\mathrm{E}_t[N_T] + 25 \cdot \sigma_t(N_T)$. Figure \ref{fig:Hawkes_CDF} and \ref{fig:Hawkes_PMF} show the CDF and PMF of $N_T - N_t$ conditional on $\mathcal{F}_t$, using 128 Fourier expansion terms for $L$ and $1024$ terms for $N$ and the sharpened raised cosine filter. A large number of Fourier terms is deliberately chosen to visualize the steps in the CDF.
 
\begin{figure}[H]
	\centering
	\includegraphics[scale=0.66]{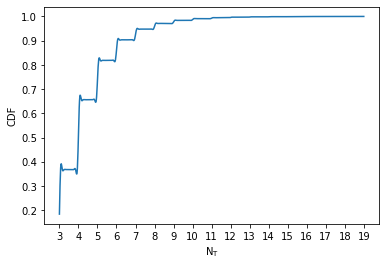}
	\caption{The COS-recovered conditional CDF of $N_T$} 
	\label{fig:Hawkes_CDF}
\end{figure}

\begin{figure}[H]
	\centering
	\includegraphics[scale=0.66]{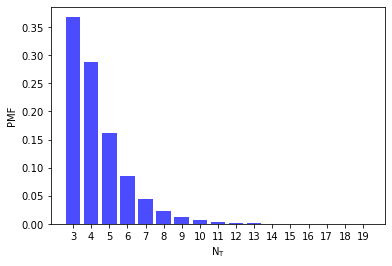}
	\caption{The COS-recovered conditional PMF of $N_T$} 
	\label{fig:Hawkes_PMF}
\end{figure}

The total computational time is about 2.5 seconds on an Intel(R) Core(TM) i5-10210U CPU, including the time spent solving the ODEs in Proposition \ref{prop:Hawkes}.

To benchmark the accuracy of the discrete COS method, we compare the COS implied first moment of $N_T$ conditional on $\mathcal{F}_t$ against its analytical solution. An excellent match is observed. The residual error stops decreasing beyond 128 terms, which can attributed to the truncation range error and the numerical error in solving the ODEs.
\begin{table}[H]
	\centering
	\begin{tabular}{|c|c|c|c|c|c|c|}
		\hline
		Number of Fourier terms & 32 & 64 & 128 & 256 & 512 & 1024  \\
		\hline
		Absolute error & 2.8e-2 & 6.3e-4 & 6.5e-9 & 1.3e-9  & 1.3e-9 &  1.3e-9    \\
		\hline
	\end{tabular}
	\caption{Error convergence of $\mathrm{E}\left(N_T|\mathcal{F}_t\right)$ with respect to the number of expansion terms. }
\end{table}

\section{Conclusion}\label{sec:conclusion}
In this paper we provide the proof of convergence of the semi-analytical COS formula for discrete distributions and extend it to the bi-variate discrete distributions. We also analyze the error convergence rate, leading to improved convergence results for several classical spectral filters.

The established theoretical convergence, supported by numerical examples, highlights the potential of using the COS method to find efficient semi-analytical solutions to problems involving discrete probability distributions, such as approximating CDFs and computing moments.

\newpage
\bibliographystyle{siamplain}
\bibliography{references}

\end{document}